\documentclass[10pt]{amsart}
\usepackage{amssymb,amsthm,amsmath,amsfonts}
\usepackage{graphicx}
\usepackage{epstopdf}
\usepackage{commath}
\usepackage{hyperref}
\usepackage[hypcap]{caption}
\def\R{{\mathbb R}}
\def\Rr{{\mathcal R}}
\def\C{{\mathbb C}}

\def\Z{{\mathbb Z}}

\def\Tr{\;\mathrm{Tr}\,}
\def\<{\langle}
\def\>{\rangle}
\def\P{{\mathbb P}}

\def\U{{\mathbb U}}
\def\V{{\mathbb V}}

\def\I{\mathbb I}

\def\X{\mathbb X}

\def\Y{\mathbb Y}

\newcommand{\be}{\begin{equation}}
\newcommand{\ee}{\end{equation}}
      \newtheorem{theorem}{Theorem}[section]

       \newtheorem{lemma}[theorem]{Lemma}
       \newtheorem{remark}{Remark}[section]
\newtheorem{definition}{Definition}[section]

\title{On the Matsumoto--Yor property in free probability}
\author[K. Szpojankowski]{Kamil Szpojankowski}
\address[K. Szpojankowski]{Department of Pure Mathematics, University of Waterloo\\
Waterloo, Ontario N2L 3G1, Canada\\
and
Wydzia\l{} Matematyki i Nauk Informacyjnych\\
Politechnika Warszawska\\
ul. Koszykowa 75\\
00-662 Warszawa, Poland.}
\email{k.szpojankowski@mini.pw.edu.pl}



\begin{document}
\begin{abstract}
	We study the Matsumoto--Yor property in free probability. We prove that the limiting empirical eigenvalue distribution of the GIG matrices and the Marchenko--Pastur distribution have the free Matsumoto--Yor property. Finally we characterize these distributions by a regression properties in free probability. 
\end{abstract}
\maketitle
\section{Introduction}
This paper is devoted to the study of so called Matsumoto--Yor property analogues in free probability. In \cite{MatsumotoYor} Matsumoto and Yor noted an interesting independence property of the Generalized Inverse Gaussian (GIG) law and the Gamma law, namely if $X$ is distributed $GIG(-p,a,b)$, where $p,a,b>0$ and $Y$ has Gamma distribution  $G(p,a)$ then 
\begin{align*}
U=\frac{1}{X+Y}\qquad\mbox{and}\qquad V=\frac{1}{X}-\frac{1}{X+Y}
\end{align*}
are independent and distributed $GIG(-p,b,a)$ and  $G(p,b)$ respectively. 

By a Generalized Inverse Gaussian distribution $GIG(p,a,b)$, $p\in\mathbb{R},$ $a,b>0$ we understand here a probability measure given by the density
\begin{align*}
f(x)=\frac{1}{K(p,a,b)}x^{p-1}e^{-ax-\frac{b}{x}}\mathit{I}_{(0,+\infty)}(x).
\end{align*}
Where $K(p,a,b)$ is a normalization constant.
A Gamma distribution $G(p,a)$, $p,a>0$ is given by the density
\begin{align*}
g(x)=\frac{a^p}{\Gamma(p)}x^{p-1}e^{-ax}\mathit{I}_{(0,+\infty)}(x).
\end{align*}

Next in \cite{LetacWesMY} Letac and Weso\l{}owski proved that the above independence property is a characterization of GIG and Gamma laws, i.e. for independent $X$ and $Y$ random variables $U$ and $V$ defined as above are independent if and only if $X$ has $GIG(-p,a,b)$ distribution and $Y$ has $G(p,a)$ distribution, for some $p,a,b>0$. In the same paper authors extend this result for real symmetric matrices.

It is remarkable that many characterizations of probability measures by independence properties have parallel results in free probability. A basic example of such analogy is Bernstein's theorem which says that for independent random variables $X$ and $Y$, variables $X+Y$ and $X-Y$ are independent if and only if $X$ and $Y$ are Gaussian distributed. The free probability analogue proved in \cite{NicaChar} gives a characterization of Wigner semicircle law by similar properties where the independence assumptions are replaced by freeness assumptions. Another well studied example of such analogy is Lukacs theorem which characterizes the Gamma law by independence of $X+Y$ and $X/(X+Y)$ for independent $X$ and $Y$ (cf. \cite{Lukacs}). The free analogue  (see \cite{BoBr2006,SzpLukacsProp}) turns out to characterize a Marchenko--Pastur law which is also known as the free Poisson distribution. 

In this paper we will be particularly interested in characterizations of probability measures by regression assumptions. An example of such characterization are well known Laha--Lukacs regressions \cite{LahaLukacs} which characterize so called Meixner distributions by assumption that for independent $X$ and $Y$ first conditional moment of $X$ given by $X+Y$ is a linear function of $X+Y$ and second conditional moment of the same type is a quadratic function of $X+Y$. In \cite{BoBr2006,EjsmontLL} authors studied free analogues of Laha-Lukacs regressions, which characterize free Miexner distribution defined in \cite{anshelevich}. This results motivated a more intensive study of regression characterizations in free probability (see \cite{EjsmontTD,EjsmontFranzSzpojankowski:2015:convolution,SzpojanWesol})

The aim of this paper is to settle a natural question if also the Matsumoto--Yor property possesses a free probability analogue. The answer turns out to be affirmative. We prove a free analogue of the regression version of this property proved in \cite{WesoloMY}, where instead of assuming independence of $U$ and $V$, constancy of regressions of $V$ and $V^{-1}$ given $U$ is assumed. The role of the Gamma law is taken as in the case of Lukacs theorem by the Marchenko--Pastur distribution, the free probability analogue of the GIG distribution turns ot to be a measure defined in \cite{Feral}, where it appears as a limiting empirical spectral distribution of complex GIG matrices. The proof of the main result partially relies on the technique which we developed in our previous papers \cite{SzpojanWesol,SzpDLRNeg}, however some new ideas were needed. As a side product we prove that the free GIG and the Marchenko--Pastur distributions have a free convolution property which mimics the convolution property of the classical GIG and Gamma distributions.

The paper is organized as follows, in Section 2 we give a brief introduction to free probability and we investigate some basic properties of the free GIG distribution, Section 3 is devoted to prove that the Marchenko--Pastur and the free GIG distribution have a freeness property which is an analogue of the classical Matsumoto--Yor property. Section 4 contains the main result of this paper which is a regression characterization of the Marchenko--Pastur and the free GIG distributions.
\section{Preliminaries} 
In this section we will give a short introduction to free probability theory, which is necessary to understand the rest of the paper. Next we give the definitions of the Marchenko--Pastur and the free GIG distribution. Since the free GIG appears for the first time in the context of free probability, we study some basic properties of this distributions.
\subsection{Non-commutative probability space and distributions}
The following section is a brief introduction to free probability theory, which covers areas needed in this paper. A comprehensive introduction to this subject can be found in \cite{NicaSpeicherLect,VoiDykNica}.

By a (non-commutative) $\star$-probability space we understand a pair $\left(\mathcal{A},\varphi\right)$, where $\mathcal{A}$ is a unital $\star$-algebra, and $\varphi:\mathcal{A}\to\mathbb{C}$ is a linear functional such that $\varphi\left(\mathit{1}_{\mathcal{A}}\right)=1$, moreover we assume that $\varphi$ is faithful, normal, tracial and positive. Elements of $\mathcal{A}$ are called random variables.

In this paper we will consider only compactly supported probability measures on the real line. Such measures are uniquely determined by the sequence of moments. Thus for a self-adjoint random variable $\X\in\mathcal{A}$ we can define the distribution of $\X$ as a unique, real, compactly supported probability measure $\mu$ for which we have
\begin{align*}
\varphi\left(\X^n\right)=\int t^n d\mu(t),
\end{align*}
if such measure exists.\\
The joint distribution of random variables is also identified with the collection of joint moments. More precisely for an n-tuple of random variables $\left(\X_1,\ldots,\X_n\right)$, for $n\geq 2$ by the joint distribution we understand the linear functional $\mu_{\left(\X_1,\ldots,\X_n\right)}:\mathbb{C}\,\langle x_1,\ldots,x_n\rangle\to\mathbb{C}$, where $\mathbb{C}\,\langle x_1,\ldots,x_n\rangle$ is a space of polynomials with complex coefficients in non-commuting variables $x_1,\ldots,x_n$, such that
 \begin{align*}
 \mu_{\left(\X_1,\ldots,\X_n\right)}(P)=\varphi\left(P\left(\X_1,\ldots,\X_n\right)\right),
 \end{align*}
 for any $P\in\mathbb{C}\,\langle x_1,\ldots,x_n\rangle$.
 
When we deal with random variables, which one dimensional marginal distributions have compact support, we can understand the notion of classical independence as a rule of calculating mixed moments. In the non-commutative probability theory there are several notions of independence which for compactly supported random variables are also nothing else but rules of calculating joint moments from marginal moments. The most important notion of independence in the non-commutative probability is freeness defined by Voiculescu in \cite{VoiculescuAdd}.\\
Let $I$ be an index set. For a non-commutative probability space $(\mathcal{A},\varphi)$ we call unital subalgebras $(\mathcal{A}_i)_{i\in I}$ free if for every $n\in\mathbb{N}$ we have
\begin{align*}
\varphi(\X_1\cdot\ldots\cdot\X_n)=0,
\end{align*}
whenever:
\begin{itemize}
	\item $\varphi\left(X_i\right)=0$ for $i=1,\ldots,n$, 
	\item $\X_i\in\mathcal{A}_{j(i)}$ for $i=1,\ldots,n$,
	\item $i(1)\neq i(2) \neq\ldots \neq i(n)$.
\end{itemize} 
The last condition means that neighbouring random variables are from different subalgebras.\\
Random variables $\X$ and $\Y$ are free if algebra generated by $\{\X,\mathbb{I}\}$ is free from algebra generated by  $\{\Y,\mathbb{I}\}$.\\
If $\X$ has distribution $\nu$ and $\Y$ has distribution $\mu$ and $\X$ and $\Y$ are free we call the distribution of $\X+\Y$ the free convolution of $\nu$ and $\mu$ and denote $\nu\boxplus\mu$.\\
The definition of freeness be understand as a rule of calculating mixed moment, since subalgebras are unital, and we can always center given random variable $\X$ by subtracting $\varphi(\X)\mathbb{I}$. \\
 
We will use in this paper the notion of asymptotic freeness for random matrices.
For any $n=1,2,\ldots$, let $(\X_1^{(n)},\ldots,\X_q^{(n)})$ be a family of random variables in a non-commutative probability space $(\mathcal{A}_n,\varphi_n)$. The sequence of distributions $(\mu_{(\X_i^{(n)},\,i=1,\ldots,q)})$ converges as $n\to\infty$ to a distribution $\mu$ if $\mu_{(\X_i^{(n)},\,i=1,\ldots,q)}(P)\to \mu(P)$ for any $P\in\C\,\langle x_1,\ldots,x_q\rangle$. If additionally $\mu$ is a distribution of a family $(\X_1,\ldots,\X_q)$ of random variables in a non-commutative space $(\mathcal{A},\varphi)$ then we say that $(\X_1^{(n)},\ldots,\X_q^{(n)})$ converges in distribution to $(\X_1,\ldots,\X_q)$. Moreover, if $\X_1,\,\ldots,\X_q$ are freely independent then we say that $\X_1^{(n)},\ldots,\X_q^{(n)}$ are asymptotically free.\\
We will in particular need the following result, let $({\bf X}_N)_{N\geq 1}$ and $({\bf Y}_N)_{N\geq 1}$ be independent sequences of hermitian random matrices. Assume that for every $N$ the distribution (in classical sense) of ${\bf X}_N$ and ${\bf Y}_N$ is unitarily invariant, moreover assume that for both matrix sequences, the corresponding empirical eigenvalue distributions converge almost surely weekly. Then sequences $({\bf X}_N)_{N\geq 1}$ and $({\bf Y}_N)_{N\geq 1}$ are almost surely asymptotically free under the states $\varphi_N=1/N Tr(\cdot)$. More precisely for any polynomial $P\in\mathbb{C}\,\langle x_1,x_2\rangle$ in two non-commuting variables we have almost surely $\lim_{N\to\infty}\varphi_N\left(P({\bf X}_N,{\bf Y}_N)\right)=\varphi(P(\X,\Y))$, where distributions of $\X$ and $\Y$ are the limits of empirical eigenvalue distributions of $({\bf X}_N)$ and $({\bf X}_N)$ respectively and $\X$ and $\Y$ are free. The proof of this fact can be found in \cite{HiaiPetz} chapter 4.
\subsection{Free cumulants} For a concrete example it is usually complicated to calculate mixed moments for free random variables. So called free cumulants defined by Speicher in \cite{SpeicherNC} give a handy combinatorial description for freeness.

Let $\chi=\{B_1,B_2,\ldots\}$ be a  partition of the set of numbers $\{1,\ldots,k\}$. A partition $\chi$ is a crossing partition if there exist distinct blocks $B_r,\,B_s\in\chi$ and numbers $i_1,i_2\in B_r$, $j_1,j_2\in B_s$ such that $i_1<j_1<i_2<j_2$. Otherwise $\chi$ is called a non-crossing partition. The set of all non-crossing partitions of $\{1,\ldots,k\}$ is denoted by $NC(k)$.

For any $k=1,2,\ldots$, (joint) free cumulants of order $k$ of non-commutative random variables $\X_1,\ldots,\X_n$ are defined recursively as $k$-linear maps $\mathcal{R}_k:\mathcal{A}^k\to\C$ through equations
$$
\varphi(\Y_1\ldots\Y_m)=\sum_{\chi\in NC(m)}\,\prod_{B\in\chi}\,\mathcal{R}_{|B|}(\Y_i,\,i\in B)
$$
holding for any $\Y_i\in\{\X_1,\ldots,\X_n\}$, $i=1,\ldots,m$, and any $m=1,2,\ldots$,
with $|B|$ denoting the size of the block $B$.

Freeness can be characterized in terms of behaviour of cumulants in the following way: consider unital subalgebras $(\mathcal{A}_i)_{i\in I}$ of an algebra $\mathcal{A}$ in a non-commutative probability space $(\mathcal{A},\,\varphi)$. Subalgebras $(\mathcal{A}_i)_{i\in I}$ are freely independent iff for any $n=2,3,\ldots$ and for any $\X_j\in\mathcal{A}_{i(j)}$ with $i(j)\in I$, $j=1,\ldots,n$ any $n$-cumulant
$$
\mathcal{R}_n(\X_1,\ldots,\X_n)=0
$$
if there exists $k,l\in\{1,\ldots,n\}$ such that $i(k)\ne i(l)$.

In the sequel we will use the following formula from \cite{BozLeinSpeich} which connects cumulants and moments for non-commutative random variables
\begin{align}\label{BLS}
\varphi(\X_1\ldots\X_n)=\sum_{k=1}^n\,\sum_{1<i_2<\ldots<i_k\le n}\,\mathcal{R}_k(\X_1,\X_{i_2},\ldots,\X_{i_k})\,\prod_{j=1}^k\,\varphi(\X_{i_j+1}\ldots\X_{i_{j+1}-1})
\end{align}
with $i_1=1$ and $i_{k+1}=n+1$ (empty products are equal 1).

\subsection{Conditional expectation}
The classical notion of conditional expectation has its non-commutative counterpart in the case $(\mathcal{A},\varphi)$ is a  $W^*$-probability spaces, that is $\mathcal{A}$ is a von Neumann algebra. We say that  $\mathcal{A}$ is a von Neumann algebra if it is a weakly closed unital $*$-subalgebra of $B(\mathcal{H})$, where $\mathcal{H}$ is a Hilbert space.\\ 
If $\mathcal{B}\subset \mathcal{A}$ is a von Neumann subalgebra of the von Neumann algebra $\mathcal{A}$, then there exists a faithful normal projection from $\mathcal{A}$ onto $\mathcal{B}$, denoted by $\varphi(\cdot|\mathcal{B})$, such that $\varphi(\varphi(\cdot|\mathcal{B}))=\varphi$. This projection $\varphi(\cdot|\mathcal{B})$ is a non-commutative conditional expectation given subalgebra $\mathcal{B}$. If $\X\in \mathcal{A}$ is self-adjoint then $\varphi(\X|\mathcal{B})$ defines a unique self-adjoint element in $\mathcal{B}$. For $\X\in\mathcal{A}$ by $\varphi(\cdot|\X)$ we denote conditional expectation given von Neumann subalgebra $\mathcal{B}$ generated by $\X$ and $\I$. Non-commutative conditional expectation has many properties analogous to those of classical conditional expectation. For more details one can consult e.g. \cite{Takesaki}. Here we state two of them we need in the sequel. The proofs can be found in \cite{BoBr2006}.
\begin{lemma}\label{conexp} Consider a $W^*$-probability space $(\mathcal{A},\varphi)$.
\begin{itemize}
\item If $\X\in\mathcal{A}$ and $\Y\in\mathcal{B}$, where $\mathcal{B}$ is a von Neumann subalgebra of $\mathcal{A}$, then
\begin{align}\label{ce1}
\varphi(\X\,\Y)=\varphi(\varphi(\X|\mathcal{B})\,\Y).
\end{align}
\item If $\X,\,\Z\in\mathcal{A}$ are freely independent then
\begin{align}\label{ce2}
\varphi(\X|\Z)=\varphi(\X)\,\mathbb{I}.
\end{align}
\end{itemize}
\end{lemma}

\subsection{Functional calculus in $C^*$-algebras}
We also will use functional calculus in $C^*$-algebras, for the details we refer to chapter 1 in \cite{Takesaki} or chapter 5 in \cite{AndersonGuionnetZeitouniIntro}. We say that $\mathcal{A}$ is a $C^*$-algebra if it is a Banach space equipped with the involution $*:\mathcal{A}\to\mathcal{A}$ satisfying the following:
\begin{itemize}
	\item $\norm{\X\Y}\leq\norm{\X}\norm{\Y}$
	\item $\norm{\X^*}=\norm{\X}$
	\item $\norm{\X\X^*}=\norm{\X}^2$.
\end{itemize}  
The functional calculus in $C^{*}$-algebras says that for any normal element $\X\in\mathcal{A}$ and any function $f$ continuous on $sp(\X)$ we can define a homomorphism $\pi: C(sp(\X))\to C^{*}(\I,\X)$, where $C^{*}(\I,\X)$ is the $C^*$ algebra generated by $\X$ and $\I$. This homomorphism is such that $\pi(id)=\X$ and $\norm{f(\X)}=\norm{f}_\infty$. For any polynomial $p(z,\bar{z})$ we have $\pi(p)=p\left(\X,\X^{*}\right)$ and for $f(x)=1/x$ we have $\pi(f)=\X^{-1}$.\\
For a $C^*$-probability space $(\mathcal{A},\varphi)$, i.e. $*$-probability space where $\mathcal{A}$ is a $C^{*}$-algebra we have $|\varphi(\X)|\leq \norm{\X}$.\\
For fixed $N$ if we consider $N\times N$ matrices then they form a $C^*$--algebra with $\norm{\cdot}$ being the spectral radius.\\
Moreover if we define $\norm{\bf X}_{p}=\left(\sum_{k=1}^n \sigma_k^p\right)^{1/p}$, where $\sigma_1,\ldots,\sigma_n$ are singular values of ${\bf X}$, then we have (see \cite{AndersonGuionnetZeitouniIntro} Appendix A)
\begin{align}
\label{inq_matr}
\norm{\bf X Y}_r\leq \norm{\bf X }_p \norm{\bf X Y}_q
\end{align}
for $1/p+1/q=1/r$ where $1\leq p,q,r\leq\infty$.\\
We also have 
\begin{align}
\label{inq_tr}
\left|\frac{1}{N}Tr(\bf X )\right|\leq \norm{\bf X }_1.
\end{align}

\subsection{Analytic tools}
Now we introduce basic analytical tools used to deal with non-commutative random variables and their distributions.

For a non-commutative random variable $\X$ its $r$-transform is defined as
\begin{align}\label{rtr}
r_{\X}(z)=\sum_{n=0}^{\infty}\,\mathcal{R}_{n+1}(\X)\,z^n.
\end{align}
In \cite{VoiculescuAdd} it is proved that $r$-transform of a random variable with compact support is analytic in a neighbourhood of zero.  From properties of cumulants it is immediate that for $\X$ and $\Y$ which are freely independent
\begin{align}\label{freeconv}
r_{\X+\Y}=r_{\X}+r_{\Y}.
\end{align}
This relation explicitly (in the sense of $r$-transform) defines free convolution of $\X$ and $\Y$.
If $\X$ has the distribution $\mu$, then often we will write $r_{\mu}$ instead of $r_{\X}$.

The Cauchy transform of a probability measure $\nu$ is defined as
$$
G_{\nu}(z)=\int_{\R}\,\frac{\nu(dx)}{z-x},\qquad \Im(z)>0.
$$
To get the measure form the Cauchy transform one can use the Stieltjes inversion formula
\begin{align*}
d\nu(t)=-\frac{1}{\pi}\lim_{\varepsilon\to 0^+}\Im G_\nu(t+i\varepsilon)
\end{align*}
The Cauchy transforms and $r$-transforms are related by
\begin{align}\label{Crr}
G_{\nu}\left(r_{\nu}(z)+\frac{1}{z}\right)=z,
\end{align}
for $z$ in neighbourhood of $0,$ and
\begin{align}\label{Crr2}
r_{\nu}(G_{\nu}(z))+\frac{1}{G_{\nu}(z)}=z,
\end{align}
for $z$ in a neighbourhood of infinity.

Finally we introduce the moment transform $M_{\X}$ of a random variable $\X$,
\begin{align}\label{mgf}
M_{\X}(z)=\sum_{n=1}^{\infty}\,\varphi(\X^n)\,z^n.
\end{align}
We will need the following lemma proved in \cite{SzpDLRNeg}.
\begin{lemma}
	\label{lem_cum}
	Let $\V$ be compactly supported, invertible non-commutative random variable. Define $C_n=\mathcal{R}_{n}\left(\V^{-1},\V,\ldots,\V\right)$, and $C(z)=\sum_{i=1}^{\infty}C_iz^{i-1}$. Then for $z$ in neighbourhood of $0$ we have
	\be 
	\label{lem_1}
	C(z)=\frac{z+C_1}{1+zr(z)},
	\ee
	where $r(z)$ is $R$-transform of $\V$. In particular,
	\be 
	\label{lem_2}
	C_2=1-C_1\mathcal{R}_1(\V),\ \ C_n=-\sum_{i=1}^{n-1}C_i\mathcal{R}_{n-i}(\V),\, n\geq 2
	\ee
\end{lemma}

\subsection{Marchenko--Pastur distribution}
A non-commutative random variable $\X$ is said to be free-Poisson variable if it has Marchenko--Pastur (or free-Poisson) distribution $\nu=\nu(\lambda,\gamma)$ defined by the formula
\begin{align*}
\nu=\max\{0,\,1-\lambda\}\,\delta_0+\lambda \tilde{\nu},
\end{align*}
where $\lambda\ge 0$ and the measure $\tilde{\nu}$, supported on the interval $(\gamma(1-\sqrt{\lambda})^2,\,\gamma(1+\sqrt{\lambda})^2)$, $\gamma>0$, has the density (with respect to the Lebesgue measure)
$$
\tilde{\nu}(dx)=\frac{1}{2\pi\gamma x}\,\sqrt{4\lambda\gamma^2-(x-\gamma(1+\lambda))^2}\,dx.
$$
The parameters $\lambda$ and $\gamma$ are called the rate and the jump size, respectively.

It is easy to see that if $\X$ is free-Poisson, $\nu(\lambda,\alpha)$, then $\mathcal{R}_n(\X)=\alpha^n\lambda$, $n=1,2,\ldots$. Therefore its $r$-transform has the form
$$
r_{\nu(\lambda,\alpha)}(z)=\frac{\lambda\alpha}{1-\alpha z}.
$$
\subsection{Free GIG distribution}
The measure which plays the role of the GIG distribution in the free probability analogue of the Matsumoto--Yor property turns out to be known as an almost sure weak limit of empirical eigenvalue distribution of GIG matrices (see \cite{Feral}). We will refer to this distribution as a free Generalized Inverse Gaussian $(fGIG)$ distribution and we adopt the definition from \cite{Feral}.
\begin{definition}
By the free Generalized Inverse Gaussian distribution we understand the measure $\mu=\mu(\lambda,\alpha,\beta)$, where $\lambda\in\mathbb{R}$ and $\alpha,\beta>0$ which is compactly supported on the interval $[a,b]$ with the density
\begin{align*}
\mu(dx)=\frac{1}{2\pi}\sqrt{(x-a)(b-x)} \left(\frac{\alpha}{x}+\frac{\beta}{\sqrt{ab}x^2}\right)dx,
\end{align*}
where $0<a<b$ are solution of
\begin{align}
\label{GIGcoef1}
1-\lambda+\alpha\sqrt{ab}-\beta\frac{a+b}{2ab}=&0\\
\label{GIGcoef2}
1+\lambda+\frac{\beta}{\sqrt{ab}}-\alpha\frac{a+b}{2}=&0.
\end{align}
\end{definition}
The Cauchy transform of a $fGIG$ distribution was also calculated in \cite{Feral}, for $\X$ distributed $\mu(\lambda,\alpha,\beta)$ its Cauchy transform is given by the formula
\begin{align*}
G_\X(z)=\frac{\alpha z^2-(\lambda-1)z-\beta-\left(\alpha z-\frac{\beta}{\sqrt{ab}}\right)\sqrt{(z-a)(z-b)}}{2z^2}.
\end{align*}
where $a,b$ are as in the definition.\\
The next lemma gives the $R$-transform of the free GIG distribution.
\begin{lemma}
Let $\X$ have $fGIG$ distribution $\mu(\lambda,\alpha,\beta)$ then
\begin{align}
\label{RtrGIG}
r_\X(z)=\frac{-\alpha+z(\lambda+1)+\sqrt{(\alpha+z(\lambda-1))^2-4\beta z(z-\alpha)(z-\gamma)}}{2z(\alpha-z)},
\end{align}
where
\begin{align}
\label{constC}
\gamma=\frac{\alpha^2 a b+\frac{\beta^2}{ab}-2\alpha\beta\left(\frac{a+b}{\sqrt{ab}}-1\right)-(\lambda-1)^2}{4\beta}.
\end{align}
\end{lemma}
\begin{proof}
First we define $K_\X(z)=r_\X(z)+1/z$, then by equation \eqref{Crr}, we have
\begin{align*}
\frac{\alpha K_\X^2(z)-(\lambda-1)K_\X(z)-\beta-\left(\alpha K_\X(z)+\frac{\beta}{\sqrt{ab}}\right)\sqrt{(K_\X(z)-a)(K_\X(z)-b)}}{2K_\X^2(z)}=z.
\end{align*}
After simple transformations and taking square of both sides we get
\begin{align}
\label{inv1}
\left(\alpha K_\X(z)-\frac{\beta}{\sqrt{ab}}\right)^2(K_\X(z)-a)(K_\X(z)-b)=\left((\alpha-2z) K_\X^2(z)-(\lambda-1)K_\X(z)-\beta\right)^2.
\end{align}
Next we expand booth sides and compare coefficients of the $K_\X^i$ for $i=0,1,\ldots,4$. We see that the coefficient of constant on booth sides are equal to $\beta^2$. The coefficients of $K_\X(z)$ are equal on both sides of the equation if
\begin{align*}
2\beta(\lambda-1)=2\alpha\beta\sqrt{ab}-\beta^2\left(\frac{a+b}{ab}\right),
\end{align*}
which is equivalent to \eqref{GIGcoef1}.\\
The above observation and equations \eqref{GIGcoef1} and \eqref{GIGcoef2} allow us to writhe equation \eqref{inv1} as
\begin{align*}
&(\alpha^2-(\alpha-2z)^2) K_\X^4(z)-4(\alpha+z(\lambda-1)) K_\X^3(z)+\\&\left(2\beta  (\alpha-2z)+ \alpha^2ab+\frac{\beta^2}{ab}-2\alpha\beta\left(\frac{a+b}{\sqrt{ab}}\right)-(\lambda-1)^2\right) K_\X^2(z)=0
\end{align*}
Solving the above equation for $K_\X(z)$ we see that $K_\X(z)=0$ is a double root, but in such case we would have $r_\X(z)=-1/z$ which is impossible for a compactly supported random variable. The remaining roots are
\begin{align*}
K_\X^\pm(z)=\frac{(\alpha+z(\lambda-1))\pm\sqrt{(\alpha+z(\lambda-1))^2-4\beta z(z-\alpha)(z-\gamma)}}{2z(\alpha-z)}.
\end{align*}
where $\gamma$ is as in \eqref{constC}.\\
This means that the $R$-transform is of the following form
\begin{align*}
r_\X(z)=\frac{-\alpha+z(\lambda+1)+\sqrt{(\alpha+z(\lambda-1))^2-4\beta z(z-\alpha)(z-\gamma)}}{2z(\alpha-z)},
\end{align*}
and we choose this branch of the square root for which $r$ is analytic at the origin.
\end{proof}
The following elementary properties of the free GIG and free Poisson distributions will be useful for us. Note that this are analogues of known properties of classical GIG and Gamma distributions (see \cite{LetacWesMY}).
\begin{remark}
\label{GIGPoissConv}
Let $\X$ and $\Y$ be free, $\X$ free GIG distributed $\mu(-\lambda,\alpha,\beta)$ and $\Y$ free Poisson distributed $\nu(1/\alpha,\lambda)$ respectively, for $\alpha,\beta,\lambda>0$ then $\X+\Y$ is free GIG distributed $\mu(\lambda,\alpha,\beta)$. 
\end{remark}
\begin{proof}
It follows from a straightforward calculation on the $R$-transforms. If we consider the function under the square root in \eqref{RtrGIG} as a polynomial in $\lambda$ (taking into account the definition of the constant $\gamma$) then the linear part in $\lambda$ cancels, so it is an even function of $\lambda$.
\end{proof}
\begin{remark}
If $\X$ has the free GIG distribution $\mu(\lambda,\alpha,\beta)$ then $\X^{-1}$ has the free GIG distribution $\mu(-\lambda,\beta,\alpha)$.
\end{remark}
\begin{proof}
This follows from elementary calculation on the density of the free GIG distribution. 
\end{proof}
\begin{remark}
In the case $\beta=0$ and $\alpha>0$ the fGIG distribution is a Marchenko--Pastur distribution $\nu(1/\alpha,\lambda)$. When $\alpha=0$ and $\beta>0$, then fGIG distribution is the distribution of an inverse of a Marchenko--Pastur distribution.
\end{remark}
\section{The free Matsumoto--Yor property}
In the paper \cite{MatsumotoYor} authors noted an interesting independence property, namely if $X$ and $Y$ are independent random variables distributed $GIG(-p,a,b)$ and $G(p,a)$ respectively for $p,a,b>0$, then random variables 
\begin{align*}
U=\frac{1}{X+Y}\qquad V=\frac{1}{X}-\frac{1}{X+Y}
\end{align*}
are independent, $U$ is distributed $GIG(p,b,a)$ and $V$ is distributed $G(p,b)$. \\
Later on in \cite{LetacWesMY} the following theorem which characterizes the distributions GIG and Gamma by the Matsumoto--Yor property was proved.
\begin{theorem}
	\label{MYproperty}
	Let $X$ and $Y$ be real, positive, independent, non-degenerated random variables. If $U=\frac{1}{X+Y}$ and $V=\frac{1}{X}-\frac{1}{X+Y}$ are independent then $Y$ has the gamma distribution with parameters $(p,a)$ and $X$ has the GIG distribution with parameters $(-p,a,b)$.
\end{theorem}
In the same paper the authors consider the Matsumoto--Yor property on real, symmetric, positive definite matrices. For our purposes more appropriate are complex Hermitian positive definite matrices, the Matsumoto--Yor property on general symmetric cones, in particular on Hermitian positive definite matrices was recently considered in \cite{KoloMY}. To state the Matsumoto--Yor property on the cone of Hermitian positive definite matrices we need to define the Wishart and GIG distribution on this cone. In the following by the $\mathcal{H}_N^+$, we will denote the cone of Hermitian positive definite matrices of size $N\times N$, where $N\geq 1$.
\begin{definition}
The GIG distribution with parameters ${\bf A},{\bf  B}\in \mathcal{H}_N^+$ and $\lambda\in \mathbb{R}$ is defined by the density
\begin{align*}
f({\bf x})=\frac{\det({\bf x})^{\lambda-N}}{K_N({\bf A},{\bf B},\lambda)}e^{-\Tr\left({\bf A} {\bf x}+{\bf B} {\bf x^{-1}}\right)}\mathit{1}_{\mathcal{H}_N^+}({\bf x}).
\end{align*}
where $K({\bf A},B,\lambda)$ is a normalization constant. We will denote this distribution by $GIG(\lambda,A,B)$.
\end{definition}
\begin{definition}
The Wishart distribution with parameters ${\bf A}\in \mathcal{H}_N^+$, $\lambda\in\{0,1,\ldots,N-1\}\cup(N-1,+\infty)$ is defined by the Laplace transform
\begin{align*}
L({\bf \sigma})=\left(\frac{\det {\bf A}}{\det ({\bf A}+{\bf \sigma})}\right)^\lambda
\end{align*} 
for any ${\bf \sigma}$ such that ${\bf A+\sigma}$ is positive definite. We denote this distribution by $W(\lambda,{\bf A})$.
\end{definition}
In the case $\lambda>N-1$ Wishart distribution has a density given by
\begin{align*}
f({\bf x})=\frac{\det({\bf x})^{\lambda-N}}{C_N({\bf A},\lambda)}e^{-\Tr\left({\bf A} {\bf x}\right)}\mathit{1}_{\mathcal{H}_N^+}({\bf x}).
\end{align*}

The Matsumoto--Yor property for Hermitian matrices (see \cite{LetacWesMY,KoloMY}) in the form that we need for the purpose of this paper says that for any $N\geq 1$ if $\bf X$ and $\bf Y$ are two independent $N\times N$ random matrices distributed $GIG(-\lambda,\alpha {\bf I}_N,\beta {\bf I}_N)$ and $W(\lambda,\alpha {\bf I}_N)$ respectively with $\lambda\in \{0,1,\ldots,N-1\}\cup(N-1,+\infty)$ and $\alpha,\beta>0$ then random matrices $\bf U=(X+Y)^{-1}$ and $\bf V=X^{-1}-(X+Y)^{-1}$ are independent and have distributions $GIG(-\lambda,\beta {\bf I}_N,\alpha {\bf I}_N)$ and $W(\lambda,\beta \bf {I}_N)$ respectively.\\

It is well known that if we take a sequence of Wishart matrices $\left({\bf Y}_N\right)$ with parameters $\lambda_N$ and $A_N$ such that ${\bf X}_N\in\mathcal{H}_N^+$, and $\lambda_N/N\to \lambda$ and $A_N=\alpha_N {\bf I}_N$ where $\alpha_N/N\to\alpha$ then the free Poisson distribution $\nu(1/\alpha,\lambda)$ is almost sure weak limit of empirical eigenvalue distributions.

A similar result holds for GIG matrices and was proved in \cite{Feral}. Namely under the assumptions that $\lambda_N/N\to\lambda\in\mathbb{R}$, $\alpha_N/N\to\alpha>0$, $\beta_N/N\to\beta>0$, the almost sure weak limit of the sequence of empirical eigenvalue distributions of matrices $GIG(\lambda_N,\alpha_N \mathit{I}_N,\beta_N  \mathit{I})$ is the free GIG distribution $\mu(\lambda,\alpha,\beta)$.\\
The above facts together with the Matsumoto--Yor property for random matrices and asymptotic freeness for unitarily invariant matrices suggest that a free analogue of the Matsumoto--Yor property should hold for free Poisson and free GIG distributions. Before we prove that this is indeed true, we will prove the following lemma which will be used in the proof of the free Matsumoto--Yor property.

\begin{lemma}
	\label{convergence}
	Let $(\mathcal{A},\varphi)$ be a $C^{*}$-probability space, assume that there exist $\X,\Y\in\mathcal{A}$ such that $\X$ and $\Y$ are free, $\X$ has the free GIG distribution $\mu(-\lambda,\alpha,\beta)$, $\Y$ has the free Poisson distribution $\nu(\lambda,1/\alpha)$ for $\lambda>1$ and $\alpha,\beta>0$.\\
	Let $\left({\bf X}_N\right)_{N\geq 1}$ be a sequence of complex matrices, where ${\bf X}_N$ is an $N\times N$ random matrix distributed $GIG(-\lambda_N,\alpha_N {\bf {\bf I}}_N,\beta_N {\bf {\bf I}}_N)$, where $\lambda_N/N\to\lambda>1, \alpha_N/N\to\alpha>0$, $\beta_N/N\to\beta>0$. Moreover let $\left({\bf Y}_N\right)_{N\geq 1}$ be a sequence of Wishart matrices where for ${\bf Y}_N$ is an $N\times N$ matrix distributed $W(\lambda_N,\alpha_N\,{\bf I}_N)$.
	
	Then for any complex polynomial $Q\in\mathbb{C}\,\langle x_1,x_2,x_3\rangle$ in three non commuting variables and any $\varepsilon>0$ we have 
	\begin{align*}
	\P\left(\left|\frac{1}{N}  Tr\left(Q\left({\bf X}_N,{\bf Y}_N,{\bf Y}_N^{-1}\right)\right)-\varphi\left(Q\left(\X,\Y,\Y^{-1}\right)\right)\right|>\varepsilon\right)\to 0,
	\end{align*} as $N\to\infty$.
\end{lemma}
\begin{proof}
	We know that the almost sure week limit of the sequence of empirical spectral distributions of ${\bf Y}_N$ is the free Poisson measure $\nu(\lambda,1/\alpha)$. Moreover it is known (see \cite{BaiSilverstein} chapter 5) that the largest and the smallest eigenvalues converge almost surely to the boundaries of the support of $\nu(\lambda,1/\alpha)$, which is equal to the interval $[1/\alpha(1-\sqrt{\lambda})^2,\,1/\alpha(1+\sqrt{\lambda})^2]$. Note that since $\lambda>1$, then the support of $\nu(\lambda,1/\alpha)$ is separated from 0. Thus for given $\delta>0$ we can find $0<a<b$ such that $supp(\nu)\subset [a,b]$ and there exists $N_0$ for which we have $\P\left(\forall_{N>N_0}\, a<\lambda_1({\bf Y}_N)<\lambda_N({\bf Y}_N)<b\right)>1-\delta$, where $\lambda_1({\bf Y}_N)$ and $\lambda_N({\bf Y}_N)$ are the smallest and the largest eigenvalue of ${\bf Y}_N$ respectively.
	
	We will prove that on the set $A=\{\omega:\,\forall_{N>N_0}\, a<\lambda_1({\bf Y}_N(\omega))<\lambda_N({\bf Y}_N(\omega))<b\}$ we have
	\begin{align}
	\label{Difference}
	\left|\frac{1}{N}  Tr\left(Q\left({\bf X}_N,{\bf Y}_N,{\bf Y}_N^{-1}\right)\right)-\varphi\left(Q\left(\X,\Y,\Y^{-1}\right)\right)\right|\leq\varepsilon,
	\end{align}
	for $N$ large enough and the result will follow.
	
	On the interval $[a,b]$ the function $f(y)=1/y$ is continuous then by the Stone-Weierstrass theorem for any $\eta>0$ there exists a polynomial $R(y)$ such that $\norm{1/y-R(y)}_{\infty}<\eta$. \\
	By the functional calculus in $C^{*}$-algebras we have that $\norm{\Y^{-1}-R(\Y)}<\eta$ and on the set $A$ for all $N>N_0$ we have $\norm{{\bf Y}_N^{-1}-R({\bf Y}_N)}_\infty<\eta$. For the details we refer to the chapter 5 in \cite{AndersonGuionnetZeitouniIntro} or to the chapter 1 in \cite{Takesaki}.\\
	Now we can rewrite \eqref{Difference} as
	\begin{align*}
	&\left|\frac{1}{N}  Tr\left(Q\left({\bf X}_N,{\bf Y}_N,{\bf Y}_N^{-1}\right)\right)-\right.\frac{1}{N}  Tr\left(Q\left({\bf X}_N,{\bf Y}_N,R({\bf Y}_N)\right)\right)\\
	&+\frac{1}{N}  Tr\left(Q\left({\bf X}_N,{\bf Y}_N,R({\bf Y}_N)\right)\right)-\varphi\left(Q\left(\X,\Y,R(\Y)\right)\right)\\
	&+\varphi\left(Q\left(\X,\Y,R(\Y)\right)\right)-\left.\varphi\left(Q\left(\X,\Y,\Y^{-1}\right)\right)\right|\\&\leq
	\left|\frac{1}{N}  Tr\left(Q\left({\bf X}_N,{\bf Y}_N,{\bf Y}_N^{-1}\right)\right)-\right.\left.\frac{1}{N}  Tr\left(Q\left({\bf X}_N,{\bf Y}_N,R({\bf Y}_N)\right)\right)\right|\\&
	+\left|\frac{1}{N}  Tr\left(Q\left({\bf X}_N,{\bf Y}_N,R({\bf Y}_N)\right)\right)\right.-\left.\varphi\left(Q\left(\X,\Y,R(\Y)\right)\right)\right|\\&
	+\left|\varphi\left(Q\left(\X,\Y,R(\Y)\right)\right)\right.-\left.\varphi\left(Q\left(\X,\Y,\Y^{-1}\right)\right)\right|=I_1+I_2+I_3
	\end{align*}
	
	Observe that by the almost sure asymptotic freeness of ${\bf X}_N$ and ${\bf Y}_N$ the term $I_2$ can be arbitrary small. Indeed, since $Q\left({\bf X}_N,{\bf Y}_N,R({\bf Y}_N)\right)$ is a polynomial in variables ${\bf X}_N$ and ${\bf Y}_N$ and $Q\left(\X,\Y,R(\Y)\right)$ is the same polynomial in $\X$ and $\Y$. Thus for $N$ large enough we have almost surely $I_2<\varepsilon/3$.

Consider now the term $I_1$, assume first that in the polynomial $Q\left({\bf X}_N,{\bf Y}_N,{\bf Y}_N^{-1}\right)$ the variable ${\bf Y}_N^{-1}$ appears only once, then using traciality we can write $Tr\left(Q\left({\bf X}_N,{\bf Y}_N,{\bf Y}_N^{-1}\right)\right)=Tr\left(P({\bf X}_N,{\bf Y}_N){\bf Y}_N^{-1}\right)$ for some polynomial $P$ in two non-commutative variables. 

Then we have

\begin{align*}
\left|\frac{1}{N}  Tr\left(Q\left({\bf X}_N,Y{\bf Y}_N,{\bf Y}_N^{-1}\right)\right)-\right.\left.\frac{1}{N}  Tr\left(Q\left({\bf X}_N,{\bf Y}_N,R({\bf Y}_N)\right)\right)\right|=\left|\frac{1}{N}  Tr\left(P({\bf X}_N,{\bf Y}_N)\left({\bf Y}_N^{-1}-R({\bf Y}_N)\right)\right)\right|
\end{align*}

Using inequalities from \eqref{inq_matr} and \eqref{inq_tr} we can write 

\begin{align*}
\left|\frac{1}{N}  Tr\left(P({\bf X}_N,{\bf Y}_N)\left({\bf Y}_N^{-1}-R({\bf Y}_N)\right)\right)\right|&\leq \norm{P({\bf X}_N,{\bf Y}_N)\left({\bf Y}_N^{-1}-R({\bf Y}_N)\right)}_1\\&\leq \norm{\left({\bf Y}_N^{-1}-R({\bf Y}_N)\right)}_{\infty}\norm{P({\bf X}_N,{\bf Y}_N)}_1,
\end{align*}
for any $\omega\in A$, hence choosing $\eta\leq\varepsilon/\left(3\norm{P({\bf X}_N,{\bf Y}_N)}_1\right)$ we have $I_2\leq\varepsilon/3$ on the set $A$.\\
Consider now the general case and assume that in polynomial $Q({\bf X}_N,{\bf Y}_n,{\bf Y}_N^{-1})$, variable ${\bf Y}_N^{-1}$ appears $n\geq 1$ times, by traciality  we may assume that $Q$ is of the form 
\begin{align*}
Q({\bf X}_N,{\bf Y}_N,{\bf Y}_N^{-1})=P_1({\bf X}_N,Y_N){\bf Y}_N^{-1}P_2({\bf X}_N,{\bf Y}_N){\bf Y}_N^{-1}\ldots P_n({\bf X}_N,{\bf Y}_N){\bf Y}_N^{-1}
\end{align*}
Then of course $Q({\bf X}_N,{\bf Y}_n,R({\bf Y}_N))$ can be written as above with ${\bf Y}_N^{-1}$ replaced by $R({\bf Y}_N)$.
To make the calculations below more legible we will denote $P_k=P_k({\bf X}_N,{\bf Y}_N)$ for $k=1,\ldots,n$.

\begin{align*}
&\left|\frac{1}{N}Tr\left(Q({\bf X}_N,{\bf Y}_n,{\bf Y}_N^{-1})-Q({\bf X}_N,{\bf Y}_n,R({\bf Y}_N))\right)\right|\\&=\left|\frac{1}{N}Tr\left(\sum_{l=1}^{n}P_1{\bf Y}_N^{-1}P_2{\bf Y}_N^{-1}\ldots P_l{\bf Y}_N^{-1}P_{l+1}R({\bf Y}_N)\ldots P_{n}R({\bf Y}_N)\right.\right.\\&\left.\left.-\sum_{l=1}^{n}P_1{\bf Y}_N^{-1}P_2{\bf Y}_N^{-1}\ldots {\bf Y}_N^{-1}P_{l-1}{\bf Y}_N^{-1}P_{l}R({\bf Y}_N)\ldots P_{n}R({\bf Y}_N)\right)\right|\\&\leq
\sum_{l=1}^{n}\left|\frac{1}{N}Tr\left(P_1{\bf Y}_N^{-1}P_2{\bf Y}_N^{-1}\ldots {\bf Y}_N^{-1}P_l({\bf Y}_N^{-1}-R({\bf Y}_N))P_{l+1}R({\bf Y}_N)\ldots P_{n}R({\bf Y}_N)\right)\right|\\
&\leq \norm{({\bf Y}_N^{-1}-R({\bf Y}_N))}_{\infty}\sum_{l=1}^{n}\norm{P_1{\bf Y}_N^{-1}P_2{\bf Y}_N^{-1}\ldots P_lP_{l+1}R({\bf Y}_N)\ldots P_{n}R({\bf Y}_N)}_{1}.
\end{align*}
Again we can make the above arbitrary small by appropriate choice of $\eta$, we can have in particular $I_2<\varepsilon/3$.

We prove similarly that $I_3<\varepsilon/3$, however we note that $I_3$ is non-random, and we do not have to restrict the calculation to the set $A$. We use basic inequalities in $C^{*}$ algebras $\left|\varphi(\Z)\right|\leq\norm{\Z}$ for any $\Z\in \mathcal{A}$ and $\norm{\Z_1\Z_2}\leq\norm{\Z_1}\norm{\Z_2}$ for any $\Z_1,\Z_2\in\mathcal{A}$. 
\end{proof}

The next theorem can be considered as a free probability analogue of the Matsumoto--Yor property. 
\begin{theorem}
	\label{freeMYproperty}
	Let $(\mathcal{A},\varphi)$ be a $C^*$- probability space.
	Let $\X$ have free GIG distribution and $\mu(-\lambda,\alpha,\beta)$ and $\Y$ have free Poisson distribution $\nu(1/\alpha,\lambda)$, with $\alpha,\beta>0$ and $\lambda>1$. If $\X$ and $\Y$ are free then 
	\begin{align*}
	\U=(\X+\Y)^{-1}\qquad\mbox{and}\qquad\V=\X^{-1}-(\X+\Y)^{-1}
	\end{align*}
	are free.
\end{theorem}
\begin{proof}
	First let us note so called Hua's identity 
	\begin{align}
	\label{Hua}
	\left(\X+\X\Y^{-1}\X\right)\left(\X^{-1}-(\X+\Y)^{-1}\right)=
	\left(\X\Y^{-1}\right)(\X+\Y)\left(\X^{-1}-(\X+\Y)^{-1}\right)=\mathbb{I}.
	\end{align}
	This means that $\V^{-1}=\X+\X\Y^{-1}\X$.
	Since we assume that $\mathcal{A}$ is a $C^*$-algebra, the subalgebras generated by invertible random variables and its inverses coincide. It is enough to prove that $\U^{-1}=\X+\Y$ and $\V^{-1}=\X+\X\Y^{-1}\X$ are free.
	
	Let us take a sequence $({\bf Y}_N)_{N\geq 1}$ of $N\times N$ Wishart matrices with parameters $A_N=\alpha_N \mathit{I}_N$ and $\lambda_N$ where $\alpha_N/N\to \alpha>0$ and $\lambda_N/N\to \lambda>1$. Moreover let $({\bf X}_{N\geq 1})$ be a sequence of complex GIG matrices with parameters $A_N$ and $\lambda_N$ as above and $B_N=\beta_N\mathit{I}_N$ where $\beta_N/N\to\beta>0$. Assume additionally that the sequences $({\bf X}_N)_{N\geq 1}$ and $({\bf Y_N})_{N\geq 1}$ are independent. Then since GIG and Wishart matrices are unitarily invariant and both posses almost sure limiting eigenvalue distributions they are almost surely asymptotically free. More precisely if we define $\mathcal{A}_N$ to be the algebra of random martices of the size $N\times N$ with entries integrable for any $1\leq p<\infty$ and consider the state $\varphi_N({\bf A})=\left(1/N\right) \Tr({\bf A})$, then for any polynomial in two non-commuting variables $P\in \mathbb{C}\,\langle x_1,x_2 \rangle$ we have almost surely
	\begin{align*}
	\lim_{N\to\infty}\varphi_N(P({\bf X_N},{\bf Y_N}))=\varphi(P(\X,\Y)),
	\end{align*}
	where $\X$ and $\Y$ are as in the statement of the theorem.
	
	By the Matsumoto--Yor property for complex matrices we have 
	\begin{align*}
		{\bf U}_N=({\bf X}_N+{\bf Y}_N)^{-1}\qquad\mbox{and}\qquad{\bf V}_N={\bf X}_N^{-1}-({\bf X}_N+{\bf Y}_N)^{-1}
	\end{align*}
	are for any $N>0$ independent and distributed $GIG(-\lambda_N,\beta_N \mathit{I}_N,\alpha_N \mathit{I}_N )$ and Wishart $W(\lambda_N,\beta_N \mathit{I}_N)$. Thus they are almost surely asymptotically free, of course the pair $({\bf U}_N^{-1}$, ${\bf V}_N^{-1})$ is also asymptotically free, let us denote the limiting pair of non-commutative free random variables by $(\tilde{\U}^{-1},\tilde{\V}^{-1})$.
	
	So we have for any polynomial  $P\in \mathbb{C}\,\langle x_1,x_2 \rangle$ almost surely 
	\begin{align}
	\label{convergence1}
	\lim_{N\to\infty}\varphi_N\left(P({\bf U}_N^{-1},{\bf V}_N^{-1})\right)=\varphi\left(P\left(\widetilde{\U}^{-1},\widetilde{\V}^{-1}\right)\right).
	\end{align}
	On the other hand for some $Q\in \mathbb{C}\,\langle x_1, x_2, x_3\rangle$ we can write the left hand side of the above equation as
	\begin{align}
	\label{convergence2}
	\lim_{N\to\infty}\varphi_N\left(P\left({\bf X}_N+{\bf Y}_N,{\bf X}_N+{\bf X}_N{\bf Y}_N^{-1}{\bf X}_N\right)\right)=\lim_{N\to\infty}\varphi_N\left(Q\left({\bf X}_N,{\bf Y}_N^{-1},{\bf Y}_N\right)\right).
	\end{align}
	Note that Lemma \ref{convergence} says that $\varphi_N\left(Q\left({\bf X}_N,{\bf Y}_N^{-1},{\bf Y}_N\right)\right)$ converges in probability to $\varphi\left(Q\left(\X,\X^{-1},\Y\right)\right)$. However by \eqref{convergence1} and \eqref{convergence2} sequence $\left(\varphi_N\left(Q\left({\bf X}_N,{\bf Y}_N^{-1},{\bf Y}_N\right)\right)\right)_{N\geq 1}$ converges almost surely so we have
	\begin{align*}
	\lim_{N\to\infty}\varphi_N\left(Q\left({\bf X}_N,{\bf X}_N^{-1},{\bf Y}_N\right)\right)=&\varphi\left(Q\left(\X,\X^{-1},\Y\right)\right)
	\\=&\varphi\left(P\left(\X+\Y,\X+\X\Y^{-1}\X\Y\right)\right)=\varphi\left(P\left(\U^{-1},\V^{-1}\right)\right).
	\end{align*}
	This means that $\U^{-1}$ and $\V^{-1}$ have all joint moments equal to joint moments of $\widetilde{\U}^{-1},\widetilde{\V}^{-1}$, which are free. This essentially means that $\U^{-1}$ and $\V^{-1}$ are free and the theorem follows.
\end{proof}
\section{A characterization by the free Matsumoto - Yor property}
In this section we prove the main result of this paper which is a regression characterization of free GIG and Marchenko--Pastur distributions. It is a free probability analogue of the result from \cite{WesoloMY} where it was proved that assumption of independence of $U$ and $V$ in Theorem \ref{MYproperty} can be replaced by constancy of regressions: $\mathbb{E}\,\left(V|U\right)$, $\mathbb{E}\,\left(V^{-1}|U\right)$.\\
Our main result is a free analogue of the above result. Note that thank to Theorem \ref{freeMYproperty} and properties of conditional expectation, we know that $\X$ free GIG distributed and $\Y$ free Poisson distributed with appropriate choice of parameters satisfy assumptions of the following theorem.
\begin{theorem}
\label{main}
Let $\left(\mathcal{A},\varphi\right)$ be a non-commutative probability space. Let $\X$ and $\Y$ be self-adjoint, positive, free, compactly supported and non-degenerated random variables. Define $\U=\left(\X+\Y\right)^{-1}$ and $\V=\X^{-1}-\left(\X+\Y\right)^{-1}$. Assume that
\begin{align}
\label{asmp}
\varphi\left(\V|\U\right)=c\mathbb{I},\\
\varphi\left(\V^{-1}|\U\right)=d\mathbb{I}\nonumber,
\end{align}
for some real constants $c,d$.\\
Then $\Y$ has the free Poisson distribution $\nu(1/\alpha,\lambda)$ with parameters $\lambda=\frac{cd}{cd-1}$ and $\alpha=\frac{\delta_0}{cd-1}$, for some $\delta_0>0$, $\X$ has free GIG distribution $\mu(-\lambda,\alpha,\beta)$, where $\beta=\frac{d}{cd-1}$.
\end{theorem}
\begin{proof}
Note that 
By Hua's identity \eqref{Hua} we have $\V^{-1}=\X+\X\Y^{-1}\X$. By this we can rewrite second equation from \eqref{asmp} as
\begin{align*}
\varphi\left(\X+\X\Y^{-1}\X|\X+\Y\right)=d\mathbb{I}.
\end{align*}
Now we multiply the above equation by $(\X+\Y)^{n}$ and apply the state $\varphi$, to both sides of the equation, which gives
\begin{align*}
\varphi\left(\X(\X+\Y)^n\right)+\varphi\left(\X\Y^{-1}\X(\X+\Y)^n\right)=d\varphi\left((\X+\Y)^n\right).
\end{align*}
Note that 
\begin{align*}
\varphi\left(\X\Y^{-1}\X(\X+\Y)^n\right)=&\varphi\left(\X\Y^{-1}(\X+\Y)^{n+1}\right)
-\varphi\left(\X\left(\X+\Y\right)^n\right)\\
=&\varphi\left(\Y^{-1}\left(\X+\Y\right)^{n+2}\right)-\varphi\left((\X+\Y)^{n+1}\right)-\varphi\left(\X\left(\X+\Y\right)^n\right).
\end{align*}
This gives us 
\begin{align}
\label{ac2}
\varphi\left(\Y^{-1}\left(\X+\Y\right)^{n+2}\right)-\varphi\left((\X+\Y)^{n+1}\right)=
d\varphi\left((\X+\Y)^n\right).
\end{align}
Multiplying the first equation from \eqref{asmp} by $(\X+\Y)^n$ we obtain
\begin{align}
\label{ac1}
\varphi\left(\X^{-1}(\X+\Y)^n\right)-\varphi\left((\X+\Y)^{n-1}\right)=c\varphi\left((\X+\Y)^n\right).
\end{align}
Let us denote for $n\geq -1$,
\begin{align*}
\alpha_n=\varphi\left((\X+\Y)^n\right),
\end{align*}
and for $n\geq 0$
\begin{align*}
\beta_n=&\varphi\left(\X^{-1}(\X+\Y)^n\right),\\
\delta_n=&\varphi\left(\Y^{-1}(\X+\Y)^n\right).
\end{align*}
Then equations \eqref{ac2} and \eqref{ac1} for $n\geq 0$ can be rewritten as
\begin{align}
\label{s1}
\beta_n-\alpha_{n-1}=&c\alpha_n,\\
\delta_{n+2}-\alpha_{n+1}=&d\alpha_n \nonumber
\end{align} 
Now we define generating functions of sequences $\left(\alpha_n\right)_{n\geq 0}$, $\left(\beta_n\right)_{n\geq 0}$, $\left(\delta_n\right)_{n\geq 0}$
\begin{align*}
A(z)=\sum_{n=0}^{\infty} \alpha_n z^n,\qquad B(z)=\sum_{n=0}^{\infty} \beta_n z^n, \qquad
D(z)=\sum_{n=0}^{\infty} \delta_n z^n.
\end{align*}
First we will determine the distribution of $\Y$ in a similar way as it was done in \cite{SzpojanWesol}.\\
System of equations \eqref{s1} is equivalent to
\begin{align}
\label{gfe1}
B(z)-z\left(A(z)+\frac{\alpha_{-1}}{z}\right)=&cA(z),\\
\frac{1}{z^2}\left(D(z)-\delta_1 z-\delta_0\right)-\frac{1}{z}\left(A(z)-1\right)=&dA(z)\nonumber.
\end{align}
Next we use formula \eqref{BLS} to express function $B$ in terms of function $A$ and the $R$-transform of $\X$.
\begin{align*}
\beta_n=&\varphi\left(\X^{-1}(\X+\Y)^n\right)=
\Rr_1\left(\X^{-1}\right)\varphi\left((\X+\Y)^n\right)\\
+&\Rr_2\left(\X^{-1},\X\right)\left(\varphi\left((\X+\Y)^{n-1}\right)+\varphi(\X+\Y)\varphi\left((\X+\Y)^{n-2}\right)
+\ldots+\varphi\left((\X+\Y)^{n-1}\right)\right)\\
+&\ldots+\\
+&\Rr_{n+1}\left(\X^{-1},\X,\ldots,\X\right).
\end{align*}
Which means that for $n\geq 0$
\begin{align*}
\beta_n=\sum_{i=1}^{n+1}\R_i\left(\X^{-1},\underbrace{\X,\ldots,\X}_{i-1}\right)\sum_{k_1+\ldots+k_i=n+1-i}\alpha_{k_1}\cdot\ldots\cdot\alpha_{k_i}.
\end{align*}
The above formula for sequence $\beta_n$ allows us to write
\begin{align*}
B(z)=&\sum_{n=0}^{\infty}\beta_n z^n=\sum_{n=0}^\infty\sum_{i=1}^{n+1} \Rr_i\left(\X^{-1},\X,\ldots,\X\right)z^{i-1}
\sum_{k_1+\ldots+k_i=n+1-i}\alpha_{k_1}z^{k_1}\cdot\ldots\cdot\alpha_{k_i}z^{k_i}\\
=&\sum_{i=1}^{\infty}\Rr_i\left(\X^{-1},\X,\ldots,\X\right)z^{i-1}\sum_{n=i-1}^{\infty}
\sum_{k_1+\ldots+k_i=n+1-i}\alpha_{k_1}z^{k_1}\cdot\ldots\cdot\alpha_{k_i}z^{k_i}\\
=&\sum_{i=1}^{\infty}\Rr_i\left(\X^{-1},\X,\ldots,\X\right)z^{i-1}A^{i}(z).
\end{align*}
Which implies
\begin{align}
\label{fB}
B(z)=A(z)C_\X(zA(z))
\end{align}
Where we denote $C_\X=\sum_{n=1}^{\infty}\Rr_n\left(\X^{-1},\underbrace{\X,\ldots,\X}_{n-1}\right)z^{n-1}$.\\
A similar argument gives us 
\begin{align}
\label{fD}
D(z)=A(z)C_\Y(zA(z)).
\end{align}
Lemma \ref{lem_cum} implies 
\begin{align}
\label{irt}
C_\X(zA(z))&=\frac{zA(z)+\beta_0}{1+zA(z)r_\X(zA(z))},\\
C_\Y(zA(z))&=\frac{zA(z)+\delta_0}{1+zA(z)r_\Y(zA(z))}.\nonumber
\end{align}
Putting equations \eqref{fB}, \eqref{fD}, \eqref{irt} into \eqref{gfe1} we get
\begin{align}
\label{gfe2}
A(z)\frac{zA(z)+\beta_0}{1+zA(z)r_\X(zA(z))}-z\left(A(z)+\frac{\alpha_{-1}}{z}\right)=&cA(z)\\
\frac{1}{z^2}\left(A(z)\frac{zA(z)+\delta_0}{1+zA(z)r_\Y(zA(z))}-\delta_1z-\delta_0\right)-
\frac{1}{z}(A(z)-1)=&dA(z)\nonumber
\end{align}
Note that $\delta_1=\varphi\left(\Y^{-1}(\X+\Y)\right)=\varphi\left(\Y^{-1}\X\right)+1$.
Moreover from second equation in \eqref{asmp} we have $\varphi(\Y^{-1}\X)=\varphi(\X\Y^{-1})=d\varphi\left((\X+\Y)^{-1}\right)$, this gives us $\delta_1=d\alpha_{-1}+1$. \\
From the first equation in \eqref{asmp} we get $\beta_0=c+\alpha_{-1}$.

Let us note that $A(z)=M_{\X+\Y}(z)=\frac{1}{z}G_{\X+\Y}\left(\frac{1}{z}\right)$. Equation \eqref{Crr2} implies that
\begin{align*}
r_{\X+\Y}(zA(z))=r_{\X+\Y}\left(G_{\X+\Y}\left(\frac{1}{z}\right)\right)=\frac{1}{z}-\frac{1}{G_{\X+\Y}\left(\frac{1}{z}\right)}=
\frac{1}{z}-\frac{1}{zA(z)}.
\end{align*}
This allows us to write $zA(z)r_{\X}(zA(z))=zA(z)r_{\X+\Y}(zA(z))-zA(z)r_\Y(zA(z))=A(z)-1-zA(z)r_\Y(zA(z)).$\\
Finally we can put this into \eqref{gfe2} and we obtain
\begin{align}
\label{gfe3}
A(z)\frac{zA(z)+c+\alpha_{-1}}{A(z)-zA(z)r_\Y(zA(z))}-z\left(A(z)+\frac{\alpha_{-1}}{z}\right)=&cA(z),\\
\frac{1}{z^2}\left(A(z)\frac{zA(z)+\delta_0}{1+zA(z)r_\Y(zA(z))}-(d\alpha_{-1}+1)z-\delta_0\right)-
\frac{1}{z}(A(z)-1)=&dA(z).\nonumber
\end{align}
Defining $h(z)=zA(z)r_\Y(zA(z))$ we have
\begin{align*}
A(z)\frac{zA(z)+c+\alpha_{-1}}{A(z)-h(z)}-z\left(A(z)+\frac{\alpha_{-1}}{z}\right)=&cA(z),\\
\frac{1}{z^2}\left(A(z)\frac{zA(z)+\delta_0}{1+h(z)}-(d\alpha_{-1}+1)z-\delta_0\right)-
\frac{1}{z}(A(z)-1)=&dA(z).
\end{align*}
After simple transformation of the above equations we arrive at
\begin{align*}
h(z)&=\frac{A(z)c(A(z)-1)}{A(z)(c+z)+\alpha_{-1}},\\
h(z)&=\frac{A(z)(zA(z)+\delta_0)}{dz^2A(z)+z(A(z)+d\alpha_{-1})+\delta_0}-1.
\end{align*}
Comparing the right hand sides of the above equations we obtain 
\begin{align*}
\frac{cA^2(z)+zA(z)+\alpha_{-1}}{cA(z)+zA(z)+\alpha_{-1}}=\frac{A(z)(zA(z)+\delta_0)}{dz^2A(z)+z(A(z)+d\alpha_{-1})+\delta_0}.
\end{align*}
After multiplying the above equation by the denominators, which is allowed in a neighbourhood of zero, we get
\begin{align*}
(zA(z)+\alpha_{-1})\left((cd-1) zA^2(z)+(dz^2+z-\delta_0)A(z)+zd\alpha_{-1}+\delta_0\right)=0.
\end{align*}
Since $A(0)=1$ and $\alpha_{-1}>0$, in some neighbourhood of zero we can divide the above equation by $zA(z)+\alpha_{-1}$, and we obtain
\begin{align}
\label{GIGmtr}
(cd-1) zA^2(z)+(dz^2+z-\delta_0)A(z)+zd\alpha_{-1}+\delta_0=0.
\end{align}
Recall that 
\begin{align*}
zA(z)r_\Y(zA(z))=&h(z)=\frac{A(z)(zA(z)+\delta_0)}{dz^2A(z)+z(A(z)+d\alpha_{-1})+\delta_0}-1\\=&
\frac{zA^2(z)+\delta_0A(z)-dz^2A(z)-z(A(z)+d\alpha_{-1})-\delta_0}{dz^2A(z)+z(A(z)+d\alpha_{-1})+\delta_0},
\end{align*}
where $\delta_0>0$.
Using \eqref{GIGmtr} we obtain
\begin{align*}
zA(z)r_\Y(zA(z))=&\frac{cdzA^2(z)}{zA^2(z)(1-cd)+A(z)\delta_0}=\frac{cdzA(z)}{\delta_0-zA(z)(cd-1)}.
\end{align*}
By the fact that $r_\Y$ is analytic in a neighbourhood of zero we obtain that \begin{align*}
r_\Y(z)=\frac{cd}{\delta_0-z(cd-1)}.
\end{align*}
Hence $\Y$ has the free Poisson distribution $\nu\left(\frac{cd-1}{\delta_0},\frac{cd}{cd-1}\right)$.

Now we will recover the distribution of $\X$, by computing the $R$-transform of $\X+\Y$. Note that the function $A$ determined by the equation \eqref{GIGmtr} is supposed to be the moment transform of $\X+\Y$. The Cauchy transform of $\X+\Y$ is equal to $G_{\X+\Y}(z)=\frac{1}{z}A\left(\frac{1}{z}\right)$, this gives us

\begin{figure}
	\includegraphics[scale=0.75]{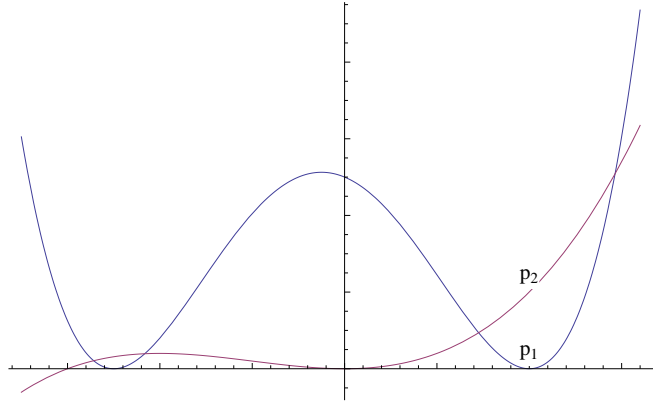}
	
	\caption{The plot of the polynomials $p_1$  and $p_2$ in the case 1. \label{Fig1}}
\end{figure}
\begin{align}
\label{Ctr_sum}
G_{\X+\Y}(z)=\frac{-d+z(z\delta_0-1)+\sqrt{(\delta_0z^2-z-d)^2-4(cd-1)z^2(d\alpha_{-1}+z\delta_0)}}{2(cd-1)z^2},
\end{align}
We have to analyze if the above function is for every set of parameters $c,d,\delta_0,\alpha_{-1}$ a Cauchy transform of a probabilistic measure. Moreover since we assume that $\X$ is positive, the support of the distribution of $\X$ must be contained in the positive half-line.\\
Let us focus on the expression under the square root 
\begin{align}
\label{usqrt}
p(z)=(\delta_0z^2-z-d)^2-4(cd-1)z^2(d\alpha_{-1}+z\delta_0).
\end{align}
Observe that the polynomial $p_1(z)=(\delta_0z^2-z-d)^2$ has two double roots which have different signs. It follows from positivity of $d$ and $\delta_0$. The polynomial $p_2(z)=4(cd-1)z^2(d\alpha_{-1}+z\delta_0)$ has a double root at the origin and a single root at $z=-\frac{d\alpha_{-1}}{\delta_0}<0$. There are three possible cases depending on the value of $\alpha_{-1}$.\\

\textbf{Case 1}. For $\alpha_{-1}$ large enough $p=p_1-p_2$ has four real roots. An example of the graph of $p_1$ and $p_2$ in this case is presented on Figure \ref{Fig1}.

One can see that the function $p$ takes negative values for $x\in(x_0,x_1)$ where $x_0$ and $x_1$ are negative roots of $p$. From the Stielties inversion formula and the form of the Cauchy transform of $\X+\Y$ \eqref{Ctr_sum}, we see that the support of the distribution of $\X+\Y$ contains the set $N=\{x\in\mathbb{R},p<0\}$, since for $x\in N$, $\Im G_{\X+\Y}(x)\neq 0$. Hence in the case 2 the interval $(x_0,x_1)$ would be contained in the support of the distribution of $\X+\Y$ which contradicts the positivity of $\X$ and $\Y$.

\textbf{Case 2}. For $\alpha_{-1}$ small enough $p$ defined by \eqref{usqrt} has two real roots and two complex roots which are conjugate, an example of the graph of the polynomials $p_1$ and $p_2$ in the case 2 is presented on Figure \ref{Fig2}.

\begin{figure}
\includegraphics[scale=0.75]{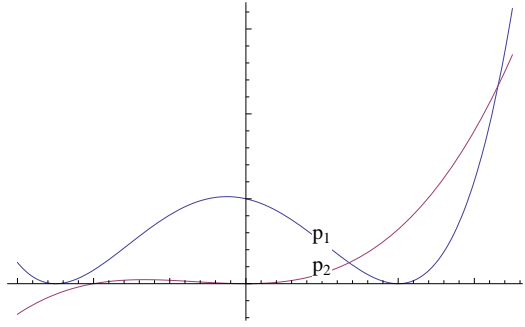}

\caption{The plot of the polynomials $p_1$  and $p_2$ in the case 2. \label{Fig2}}
\end{figure}

Recall that Cauchy transform is an analytic function from $\mathbb{C}^+$ to $\mathbb{C}^-$. In the analysed case the polynomial under the square root has a zero in  $\mathbb{C}^+$, so the function \eqref{Ctr_sum} is not analytic on $\mathbb{C}^+$ and can not be a Cauchy transform of a probabilistic measure.

\textbf{Case 3}. The remaining possibility is that $\alpha_{-1}$ is such that polynomials $p_1$ and $p_2$ are tangent for some $x_0<0$. Figure \ref{Fig3} gives an example of such situation. In this case the polynomial $p$ has two positive real roots $(0<a<b)$ and one double, real, negative root $x_0$. We will prove that in this case the function \eqref{Ctr_sum} is the Cauchy transform of free GIG distribution $\mu(cd/(cd-1),\delta_0/(cd-1),d/(cd-1)).$ \\
Let us substitute in \eqref{Ctr_sum} $c=\lambda/\beta,$ $d=\beta/(\lambda-1),$ $\delta=\alpha/(\lambda-1)$, for $\alpha,\beta>0$ and $\lambda>1$, this substitution is correct since $cd>1$ . The Cauchy transform \eqref{Ctr_sum} has the form

\begin{align*}
G_{\X+\Y}(z)=\frac{\alpha  z^2-\beta-(\lambda -1) z+\sqrt{\alpha ^2 z^4-2 \alpha  (\lambda +1) z^3+\left((\lambda -1)^2-2 \beta  (\alpha +2 \alpha_1)\right)z^2 +2 \beta  (\lambda -1) z+\beta^2}}{2 z^2},
\end{align*}
we choose the branch of the square root, such that the above function is analytic from $\mathbb{C}\,^{+}$ to $\mathbb{C}\,^{-}$.

Since polynomial under the square root has single roots $0<a<b$ and one double root $x_0<0$, Vieta's formulas for this polynomial give us
\begin{align}
\label{Vieta}
a+b+2x_0&=\frac{2(\lambda+1)}{\alpha},\\
ab+2x_0(a+b)+x_0^2&=\frac{(\lambda-1)^2-2(\alpha+2\alpha_{-1})\beta}{\alpha^2},\nonumber\\ \nonumber
2abx_0+(a+b)x_0^2&=-\frac{2\beta(\lambda-1)}{\alpha^2},\\ 
abx_0^2&=\frac{\beta^{2}}{\alpha^2}. \nonumber
\end{align}

\begin{figure}
	\includegraphics[scale=0.75]{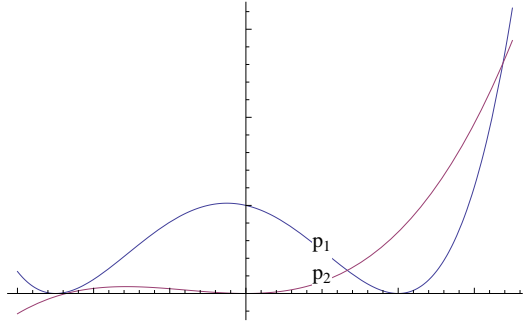}
	\caption{The plot of the polynomials $p_1$  and $p_2$ in the case 3. \label{Fig3}}
\end{figure}

We calculate $x_0$ from the last equation, since $x_0<0$ and $\alpha,a,b>0$ we get $x_0=-\frac{\beta}{\alpha\sqrt{ab}}$. From the first equation we obtain $x_0=\frac{\lambda+1}{\alpha}-\frac{a+b}{2}$. Comparing the formulas for $x_0$ we get 
\begin{align}
\label{Feral1}
1+\lambda+\frac{\beta}{\sqrt{ab}}-\alpha\frac{a+b}{2}=0.
\end{align}

Substituting $x_0=-\frac{\beta}{\alpha\sqrt{ab}}$ in the third equation from \eqref{Vieta} we obtain

\begin{align}
\label{Feral2}
1-\lambda+\alpha\sqrt{ab}-\beta\frac{a+b}{2ab}=0.
\end{align}

We can now write the function $G_{\X+\Y}$ transform as
\begin{align}
\label{Crt_sum}
G_{\X+\Y}(z)=\frac{\alpha  z^2-\beta-(\lambda -1) z+\left(\alpha z+\frac{\beta}{\sqrt{ab}}\right)\sqrt{(z-a)(z-b)}}{2 z^2}
\end{align}
with parameters satisfying \eqref{Feral1} and \eqref{Feral2} we recognize that this is the Cauchy transfrom of the free GIG distribution $\mu(\lambda,\alpha,\beta)$, where $\lambda=cd/(cd-1),$ $\alpha=\delta_0/(cd-1),$ $\beta=d/(cd-1))$. 
The second equation from \eqref{Vieta} gives the value of $\alpha_{-1}$ for which the Cauchy transform \eqref{Crt_sum} satisfies the assumptions of the case 3, one can check that $\alpha_{-1}=-\gamma$, where $\gamma$ is the constant defined in equation \eqref{constC}. 

Since the distribution of $\X+\Y$ is free GIG $\mu(\lambda,\alpha,\beta)$ and the distribution of $\Y$ is free Poisson $\nu(1/\alpha,\lambda)$ then by Remark \ref{GIGPoissConv} and freeness of $\X$ and $\Y$ it is immediate that $\X$ has the free GIG distribution $\mu(-\lambda,\alpha,\beta)$. 

\end{proof}

We can summarize the results proved in this paper by the following characterization of the free GIG and the free Poisson distributions.
\begin{remark}
Let $\X$ and $\Y$ be positive compactly supported random variables. Random variables $\U=(\X+\Y)^{-1}$ and $\V=\X^{-1}-(\X+\Y)^{-1}$ are free if and only if there exist constants $\alpha,\beta>0$ and $\lambda>1$ such that $\X$ has free GIG distribution $\mu(-\lambda,\alpha,\beta)$ and $\Y$ has free Poisson distribution $\nu(1/\alpha,\lambda)$. 
\end{remark}
\begin{proof}
	The "if" part is exactly the statement of Theorem \ref{freeMYproperty}. The "only if" part follows from Theorem \ref{main} since for free random variables conditional expectations are constant.
\end{proof}

\subsection*{Acknowledgement} The author thanks J. Weso\l{}owski for helpful comments and discussions. The author is grateful to W. Bryc for sending his notes about the free Matsumoto-Yor. The author also thanks F. Benaych-Georges for suggesting Lemma \ref{convergence} and the way to prove it. This research was partially supported by NCN grant 2012/05/B/ST1/00554.
\bibliographystyle{plain}
\bibliography{Bibl}

\end{document}